\documentclass{article}

\usepackage{graphicx}
\usepackage{indentfirst}
\usepackage{amsmath,amsfonts,amsthm,amssymb}
\usepackage{mathrsfs}
\usepackage{amscd}
\usepackage{hyperref}

\allowdisplaybreaks

\def\co{\colon\thinspace}
\DeclareMathAlphabet{\mathsfsl}{OT1}{cmss}{m}{sl}

\newcommand{\spin}{\mathrm{Spin}^c}

\newtheorem{thm}{Theorem}[section]
\newtheorem{lem}[thm]{Lemma}
\newtheorem{cor}[thm]{Corollary}
\newtheorem{prop}[thm]{Proposition}
\newtheorem{conj}[thm]{Conjecture}
\newtheorem*{thm*}{Theorem}

\theoremstyle{definition}

\newtheorem{rem}[thm]{Remark}

\begin{document}

\title{Half-integral finite surgeries on knots in $S^3$}

\author{
{Eileen LI and Yi NI}\\{\normalsize Department of Mathematics, Caltech}\\
{\normalsize 1200 E California Blvd, Pasadena, CA
91125}\\{\small\it Email\/: \quad\rm eileen.li.20854@gmail.com \quad yini@caltech.edu}}

\date{}
\maketitle

\begin{abstract}
Suppose that a hyperbolic knot in $S^3$ admits a finite surgery, Boyer and Zhang proved that the surgery slope must be either integral or half-integral, and they conjectured that the latter case does not happen. Using the correction terms in Heegaard Floer homology, we prove that if a hyperbolic knot in $S^3$ admits a half-integral finite surgery, then the knot must have the same knot Floer homology as one of eight non-hyperbolic knots which are known to admit such surgeries, and the resulting manifold must be one of ten spherical space forms. As knot Floer homology carries a lot of information about the knot, this gives a strong evidence to Boyer--Zhang's conjecture.
\end{abstract}

\section{Introduction}

Suppose that $M$ is a $3$--manifold with torus boundary, $\alpha$ is a slope on $\partial M$. Let $M(\alpha)$ be the Dehn filling along $\alpha$. If $M$ is hyperbolic, Thurston's Hyperbolic Dehn Surgery Theorem says that at most finitely many fillings are non-hyperbolic. These surgeries are called {\it exceptional surgeries}. The famous Cyclic Surgery Theorem \cite{CGLS} asserts that, if $M$ is not Seifert fibered, $\alpha,\beta$ are slopes on $\partial M$ such that both $M(\alpha)$ and $M(\beta)$ have cyclic fundamental groups, then $\Delta(\alpha,\beta)$, the distance between $\alpha,\beta$, is at most $1$. More generally, estimating the distance between any two exceptional surgery slopes is a central problem in Dehn surgery.

In \cite{BZ}, Boyer and Zhang proved that if $M$ is hyperbolic, $M(\alpha)$ has finite fundamental group (or being a finite surgery) and $M(\beta)$ has cyclic fundamental group (or being a cyclic surgery), then $|\Delta(\alpha,\beta)|\le2$. In particular, if the $\frac pq$ surgery on a hyperbolic knot $K\subset S^3$, denoted 
$S^3_{K}(\frac{p}{q})$, has a finite fundamental group, then $|q|\le2$. In fact, Boyer and Zhang made the following conjecture.

\begin{conj}\label{conj:IntFinite}
Suppose that $K\subset S^3$ is a hyperbolic knot, and that $S^3_{K}(\frac{p}{q})$ has a finite fundamental group, then $\frac pq$ must be an integer.
\end{conj}

By Perelman's resolution of the Geometrization Conjecture \cite{Pe1,Pe2,Pe3}, if a $3$--manifold has a finite fundamental group, then it is necessarily a spherical space form. In order to prove Conjecture~\ref{conj:IntFinite}, we only need to rule out {\bf T}- and {\bf I}-type spherical space forms as results of half-integer surgeries, see Section~\ref{sect:Strat} for more detail. 

In this paper, all manifolds are oriented. If $Y$ is an oriented manifold, then $-Y$ denotes the same manifold with the opposite orientation.

Let $\mathbb T$ be the exterior of the right hand trefoil, then $\mathbb T(\frac pq)$ is the manifold obtained by $\frac pq$ surgery on the right hand trefoil. It is well-known that any {\bf T}- or {\bf I}-type manifold is homeomorphic to a $\pm\mathbb T(\frac pq)$ (see Lemma~\ref{lem:SFS23}).

Let $T(p,q)$ be the $(p,q)$ torus knot, and let $[p_1,q_1;p_2,q_2]$ denote the $(p_1,q_1)$ cable of $T(p_2,q_2)$.

Our main result is the following theorem.

\begin{thm}\label{thm:TenCases}
Suppose that $K\subset S^3$ is a hyperbolic knot, and that the $\frac p2$ surgery on $K$ is a spherical space form for some odd integer $p>0$, then $K$ has the same knot Floer homology as either $T(5,2)$ or a cable of a torus knot (which must be $T(3,2)$ or $T(5,2)$), and the resulting manifold is homeomorphic to the $\frac p2$ surgery on the corresponding torus or cable knot. More precisely, the possible cases are listed in the following table:
\begin{center}
\begin{tabular}{|c|c|c|}
\hline
knot type &slope &resulting manifold\\
\hline
$T(5,2)$ &${17}/2$ &$-\mathbb T({17}/2)$\\
\hline
$T(5,2)$ &${23}/2$ &$\mathbb T({23}/3)$\\
\hline
$[11,2;3,2]$ &${43}/2$ &$\mathbb T({43}/8)$\\
\hline
$[11,2;3,2]$ &${45}/2$ &$\mathbb T({45}/8)$\\
\hline
$[13,2;3,2]$ &${51}/2$ &$\mathbb T({51}/8)$\\
\hline
$[13,2;3,2]$ &${53}/2$ &$\mathbb T({53}/8)$\\
\hline
$[19,2;5,2]$ &${77}/2$ &$-\mathbb T({77}/{12})$\\
\hline
$[21,2;5,2]$ &${83}/2$ &$\mathbb T({83}/{13})$\\
\hline
$[17,3;3,2]$ &${103}/2$ &$\mathbb T({103}/{18})$\\
\hline
$[19,3;3,2]$ &${113}/2$ &$\mathbb T({113}/{18})$\\
\hline
\end{tabular}
\end{center}
Here the first column lists the knots with which $K$ shares the same knot Floer homology.
\end{thm}

Bleiler and Hodgeson \cite{BH} have classified finite surgeries on all the iterated torus knots. The knots in the first column above are contained in their list.

The strategy of our proof is to compute the Heegaard Floer correction terms for the {\bf T}- and {\bf I}- type manifolds, then compare them with the correction terms of the half-integral surgeries on knots in $S^3$. If they match, then the knot Floer homology of the knots can be recovered from the correction terms.  

Heegaard Floer homology has been successfully used in the study of finite surgery, see, for example, \cite{OSzLens,OSzRatSurg,Greene,Doig}. The point here is that spherical space forms have the simplest possible Heegaard Floer homology, hence the information about the Heegaard Floer homology is completely contained in the correction terms. We will address this fact in more detail in Section~\ref{sect:Pre}.

Knot Floer homology tells us a lot about knots. For example, it detects the genus \cite{OSzGenus} and fiberedness \cite{Gh,NiFibred}. It is reasonable to expect that a knot with the same knot Floer homology as a knot in the above table must be the corresponding knot. Moreover, we propose the following conjecture, which would imply Conjecture~\ref{conj:IntFinite}.

\begin{conj}
Suppose $K\subset S^3$ has an L-space surgery. If all roots of its Alexander polynomial $\Delta_K(t)$ are unit roots, then $K$ is an iterated torus knot.
\end{conj}

This paper is organized as follows. In Section~\ref{sect:Pre}, we recall the basic definition and properties of the correction terms. In Section~\ref{sect:Strat}, we discuss the strategy in our proof.
In Section~\ref{sect:pLarge}, we will show that half-integral finite surgeries do not exist when $p$ is sufficiently large by concrete computations involving correction terms.
Our approach in this section is inspired by Gu \cite{Gu}.

\vspace{5pt}\noindent{\bf Acknowledgements.}\quad  The second author wishes to thank Xingru Zhang for asking the question about half-integral finite surgery and explaining the background. The second author is also grateful to Liling Gu, whose work \cite{Gu} benefits our paper a lot. The first author was supported by Caltech's Summer Undergraduate Research Fellowships program. The second author was
partially supported by an AIM Five-Year Fellowship, NSF grant
numbers DMS-1103976, DMS-1252992, and an Alfred P. Sloan Research Fellowship.  


\section{Preliminaries on Heegaard Floer homology and correction terms}\label{sect:Pre}

Heegaard Floer homology was introduced by Ozsv\'ath and Szab\'o \cite{OSzAnn1}. Given a closed oriented $3$--manifold $Y$ and a Spin$^c$ structure $\mathfrak s\in\spin(Y)$, one can define the Heegaard Floer homology $\widehat{HF}(Y,\mathfrak s),HF^+(Y,\mathfrak s),\dots$, which are invariants of $(Y,\mathfrak s)$. When $\mathfrak s$ is torsion, there is an absolute $\mathbb Q$--grading on $HF^+(Y,\mathfrak s)$. When $Y$ is a rational homology sphere, Ozsv\'ath and Szab\'o \cite{OSzAbGr} defined a {\it correction term} $d(Y,\mathfrak s)\in\mathbb Q$, which is basically the shifting of $HF^+(Y,\mathfrak s)$ relative to $HF^+(S^3)$ in the absolute grading.

The correction terms enjoy the following symmetries:
\begin{equation}\label{eq:dSymm}
d(Y,\mathfrak s)=d(Y,J\mathfrak s),\quad d(-Y,\mathfrak s)=-d(Y,\mathfrak s),
\end{equation}
where $J\co \spin(Y)\to\spin(Y)$ is the conjugation.

Suppose that $Y$ is an integral homology sphere, $K\subset Y$ is a knot. Let $Y_K(p/q)$ be the manifold obtained by $\frac pq$--surgery on $K$. Ozsv\'ath and Szab\'o defined a natural identification $\sigma\co\mathbb Z/p\mathbb Z\to\spin(Y_K(p/q))$ \cite{OSzAbGr,OSzRatSurg}. For simplicity, we often use an integer $i$ to denote the Spin$^c$ structure $\sigma([i])$, when $[i]\in\mathbb Z/p\mathbb Z$ is the congruence class of $i$ modulo $p$.

A rational homology sphere $Y$ is an {\it L-space} if $\mathrm{rank}\widehat{HF}(Y)=|H_1(Y)|$. Examples of L-spaces include spherical space forms. The information about the Heegaard Floer homology of an L-space is completely encoded in its correction terms.

Let $L(p,q)$ be the lens space obtained by $\frac{p}q$--surgery on the unknot. 
The correction terms for lens spaces can be computed inductively as in \cite{OSzAbGr}:
\begin{eqnarray}
d(S^3,0)&=&0,\nonumber\\
d(L(p,q),i)&=&-\frac14+\frac{(2i+1-p-q)^2}{4pq}-d(L(q,r),j),\label{eq:CorrRecurs}
\end{eqnarray}
where $0\le i<p+q$,
$r$ and $j$ are the reductions modulo $p$ of $q$ and $i$, respectively.

For example, using the recursion formula (\ref{eq:CorrRecurs}), we can compute
\begin{eqnarray}
d(L(p,1),i)&=&-\frac14+\frac{(2i-p)^2}{4p}\nonumber\\
d(L(p,2),i)&=&-\frac14+\frac{(2i-p-1)^2}{8p}-d(L(2,1),j)\nonumber\\
&=&\frac{(2i-p-1)^2}{8p}-\frac{1+(-1)^i}4\label{eq:CorrLp2}\\
d(L(3,q),i)&=&\left\{
\begin{array}{ll}
(\frac12,-\frac16,-\frac16),&q=1,i=0,1,2\\
(\frac16,\frac16,-\frac12),&q=2,i=0,1,2
\end{array}
\right.\label{eq:CorrL3q}\\
d(L(5,q),i)&=&\left\{
\begin{array}{ll}
(1,\frac15,-\frac15,-\frac15,\frac15),&q=1, i=0,1,2,3,4\\
(\frac25,\frac25,-\frac25,0,-\frac25),&q=2, i=0,1,2,3,4\\
(\frac25,0,\frac25,-\frac25,-\frac25),&q=3, i=0,1,2,3,4\\
(-\frac15,\frac15,\frac15,-\frac15,-1),&q=4, i=0,1,2,3,4
\end{array}
\right.\label{eq:CorrL5q}
\end{eqnarray}

Given a null-homologous knot $K\subset Y$, Ozsv\'ath--Szab\'o \cite{OSzKnot} and Rasmussen \cite{Ras} defined the knot Floer homology. The basic philosophy is, if we know all the information about the 
knot Floer homology, then we can compute the Heegaard Floer homology of all the surgeries on $K$. In particular,
if  the $\frac pq$--surgery on $K\subset S^3$ is an L-space surgery, where $p,q>0$, then the correction terms of $S^3_{K}(p/q)$ can be computed from the Alexander polynomial $\Delta_K(T)$ of $K$ as follows.
 
Suppose 
$$\Delta_K(T)=\sum_ia_it^i.$$
Define a sequence of integers
$$t_i=\sum_{j=1}^{\infty}ja_{i+j},\quad i\ge0.$$
then $a_i$ can be recovered from $t_i$ by
\begin{equation}\label{eq:a_i}
a_i=t_{i-1}-2t_i+t_{i+1},\quad\text{for }i>0.
\end{equation}

If $K$ admits an L-space surgery, then one can prove \cite{OSzRatSurg,Ras}
\begin{equation}\label{eq:tsProperties}
t_s\ge0,\quad t_s\ge t_{s+1}\ge t_s-1,\quad t_{g(K)}=0.
\end{equation}
Moreover, the following proposition holds.

\begin{prop}\label{prop:Corr}
Suppose the $\frac pq$--surgery on $K\subset S^3$ is an L-space surgery, where $p,q>0$. Then for any $0\le i\le p-1$ we have
$$d(S^3_{K}(p/q),i)=d(L(p,q),i)-2\max\{t_{\lfloor\frac{i}q\rfloor},t_{\lfloor\frac{p+q-1-i}q\rfloor}\}.$$
\end{prop}

This formula is contained in Ozsv\'ath--Szab\'o \cite{OSzRatSurg} and Rasmussen \cite{Ras}. More general version of this formula can be found in Ni--Wu \cite{NiWu}.

\begin{lem}\label{lem:J}
Suppose $i$ is an integer satisfying $0\le i<p+q$, then $J(\sigma([i]))$ is represented by $p+q-1-i$. 
\end{lem}
\begin{proof}
We only need to examine our result for surgeries on the unknot, as this is a homological statement. 

As in the proof of \cite[Proposition~4.8]{OSzAbGr}, there is a two-handle addition cobordism $X$ from $-L(q,r)$ to $-L(p,q)$, 
where $r$ is the reduction of $p$ modulo $q$.
Let $i\in\{0,1,\dots,p+q-1\}$ be a number. Let $r$ and $j$ be the reduction of $p$ and $i$ modulo $q$. The proof of \cite[Proposition~4.8]{OSzAbGr} shows that
there is a Spin$^c$ structure $\mathfrak s_z(\psi_i)$ such that its restriction on $-L(p,q)$ is represented by $i$ and its restriction on $-L(q,r)$ is represented by $j$. Moreover, 
$$\langle c_1(\mathfrak s_z(\psi_i)),H\rangle=2i+1-p-q$$
for a generator $H$ of $H_2(X)$.

Now we choose $i_1,i_2\in\{0,1,\dots,p+q-1\}$ such that $i_1+i_2=p+q-1$, and let $j_1,j_2$ be the reductions of $i_1,i_2$ modulo $q$. 
We have 
\begin{equation}\label{eq:c1neg}
\langle c_1(\mathfrak s_z(\psi_{i_1})),H\rangle=-\langle c_1(\mathfrak s_z(\psi_{i_2})),H\rangle.
\end{equation}

We claim that $$H^2(X)\cong H_2(X,\partial X)\cong \mathbb Z.$$
In fact, let $Y_1=Y_3=-L(q,r)$, $Y_2=-L(p,q)$. We have exact sequences:
$$0\to H_2(X,Y_i)\to H_2(X,\partial X)\to H_1(Y_{i+1})\to H_1(X,Y_i),\quad i=1,2,$$
Since $X$ is a $2$--handle cobordism, $$H_2(X,Y_1)\cong H_2(X,Y_2)\cong \mathbb Z,\quad H_1(X,Y_1)\cong H_1(X,Y_2)\cong 0.$$ So the above exact sequences become 
$$0\to\mathbb Z\to H_2(X,\partial X)\to\mathbb Z/p\mathbb Z\to0,$$
$$0\to\mathbb Z\to H_2(X,\partial X)\to\mathbb Z/q\mathbb Z\to0.$$
As $\gcd(p,q)=1$, it is easy to see $H_2(X,\partial X)\cong\mathbb Z$. This finishes the proof of the claim.

It follows from the claim and (\ref{eq:c1neg}) that $\psi_{i_2}=J\psi_{i_1}$. Hence $i_2$ represents $J(\sigma([i_1]))$.
\end{proof}

\begin{cor}\label{cor:Center}
If $p$ is odd, then $J\co \mathbb Z/p\mathbb Z\to\mathbb Z/p\mathbb Z$ has a unique fixed point:
$$C(p,q)=\left\{
\begin{array}{ll}
\displaystyle\frac{p+q-1}2, &\text{if $q$ is even},\\
&\\
\displaystyle\frac{q-1}2, &\text{if $q$ is odd}.
\end{array}
\right.$$ 
\end{cor}


\section{The strategy of our proof}\label{sect:Strat}

By the Geometrization theorem \cite{Pe1,Pe2,Pe3}, if a $3$--manifold has a finite fundamental group, then the manifold must be a spherical space form. Besides $S^3$, there are five types of spherical space forms: {\bf C, D, T, O, I}.
The {\bf C}-type  manifolds are the lens spaces with cyclic fundamental groups; the {\bf D}-type  manifolds are Seifert fibered spaces over the orbifold $S^2(2,2,n)$ with dihedral type fundamental groups; the {\bf T}-type  manifolds are Seifert fibered spaces over the orbifold $S^2(2,3,3)$ with tetrahedral type fundamental groups; the {\bf O}-type  manifolds are Seifert fibered spaces over the orbifold $S^2(2,3,4)$ with octahedral type fundamental groups; the {\bf I}-type  manifolds are Seifert fibered spaces over the orbifold $S^2(2,3,5)$ with icosahedral type fundamental groups. 

It follows from the Cyclic Surgery Theorem \cite{CGLS} that {\bf C}-type manifolds cannot be obtained by half-integral surgeries on hyperbolic knots. The {\bf D}- and $\bf O$-type manifolds have even order $H_1$, so they cannot be obtained from half-integral surgery. We only need to consider {\bf T}- and {\bf I}-type manifolds.

\begin{lem}\label{lem:SFS23}
Any {\bf T}-type  manifold is homeomorphic to $\pm\mathbb T(\frac{6n\pm3}n)$ for some positive integer $n$ with $\gcd(n,3)=1$. Any {\bf I}-type  manifold is homeomorphic to $\pm \mathbb T(\frac{6n\pm5}n)$ for some positive integer $n$ with $\gcd(n,5)=1$. 
\end{lem}
\begin{proof}
Suppose $Y$ is a {\bf T}-type manifold, then it is Seifert fibered over the base orbifold $S^2(2,3,3)$. Removing the neighborhood of a multiplicity $3$ singular fiber, we get a Seifert fibered space over the orbifold $D^2(2,3)$. The classification of Seifert fibered spaces tells us that there is only one such manifold up to orientation reversal, which is the trefoil complement $\mathbb T$. So $Y$ or $-Y$ can be obtained by Dehn filling on $\mathbb T$. The same argument works for {\bf I}-type manifolds.

Now we consider the problem when we get Seifert fibered spaces with base orbifold $S^2(2,3,3)$ and $S^2(2,3,5)$ by Dehn filling on $\mathbb T$.
The regular fiber on $\partial \mathbb T$ has slope $6$, so $\frac pq$--filling will create a multiplicity $\Delta(\frac pq,6)$ fiber. To get a Seifert fibered space with base orbifold $S^2(2,3,3)$ or $S^2(2,3,5)$, we need to have $\Delta(\frac pq,6)=3$ or $5$. So $\frac pq=\frac{6n\pm3}n$ or $\frac{6n\pm5}n$ for some $n>0$. 
\end{proof}


Let $p,q>0$ be coprime integers. Using Proposition~\ref{prop:Corr}, we get
\begin{equation}\label{eq:TrSurgCorr}
d(\mathbb T(p/q),i)=d(L(p,q),i)-2\chi_{[0,q)}(i),
\end{equation}
where
$$
\chi_{[0,q)}(i)=\left\{
\begin{array}{ll}
1, &\text{when }0\le i<q,\\
&\\
0, &\text{when }q\le i<p.
\end{array}
\right.
$$

Suppose $S^3_{K}(p/2)$ is a spherical manifold, then by (\ref{eq:CorrLp2}) and Proposition~\ref{prop:Corr}
\begin{eqnarray*}
d(S^3_{K}(p/2),i)&=&d(L(p,2),i)-2\max\{t_{\lfloor\frac{i}2\rfloor},t_{\lfloor\frac{p+1-i}2\rfloor}\}\\ 
&=&-\frac{1+(-1)^i}4+\frac{(2i-p-1)^2}{8p}-2t_{\min\{\lfloor\frac{i}2\rfloor,\lfloor\frac{p+1-i}2\rfloor\}}.
\end{eqnarray*}

If $S^3_{K}(p/2)\cong\varepsilon\mathbb T(p/q)$ for $\varepsilon\in\{-1,1\}$, where ``$\cong$'' stands for orientation preserving homeomorphism, then the two sets $$\{d(S^3_K(p/2),i)|\:i\in\mathbb Z/p\mathbb Z\}, \quad\{d(\varepsilon\mathbb T(p/q),i)|\:i\in\mathbb Z/p\mathbb Z\}$$
are necessarily equal. However, the two parametrizations of $\spin$ may not be equal: they could differ by an affine isomorphism of $\mathbb Z/p\mathbb Z$. More precisely, there exists an affine isomorphism $\phi\co\mathbb Z/p\mathbb Z\to\mathbb Z/p\mathbb Z$, such that 
$$d(S^3_{K}(p/2),i)=d(\varepsilon\mathbb T(p/q),\phi(i)).$$
This map $\phi$ commutes with $J$, so it follows from Corollary~\ref{cor:Center} that $\phi(C(p,2))=C(p,q)$.
For any integer $a$,
define $\phi_a\co\mathbb Z/p\mathbb Z\to\mathbb Z/p\mathbb Z$ by
\begin{equation}\label{eq:phi}
\phi_a(i)=a\big(i-C(p,2)\big)+C(p,q).
\end{equation}

By (\ref{eq:dSymm}) and Lemma~\ref{lem:J}, $d(\mathbb T(p/q),\phi_a(i))=d(\mathbb T(p/q),\phi_{p-a}(i))$.
So we may assume 
$$d(S^3_{K}(p/2),i)=\varepsilon d(\mathbb T(p/q),\phi_a(i)),\quad \text{for any }i\in\mathbb Z/p\mathbb Z,$$ 
and for some $a$ satisfying
\begin{equation}\label{eq:aCond}
0<a<\frac p2,\quad\gcd(p,a)=1.
\end{equation}

Let
\begin{equation}\label{eq:delta_a}
\delta^{\varepsilon}_a(i)=d(L(p,2),i)-\varepsilon d(\mathbb T(p/q),\phi_a(i)).
\end{equation}
By Proposition~\ref{prop:Corr}, we should have 
\begin{equation}\label{eq:Expected}
\delta^{\varepsilon}_a(i)=2t_{\text{min}\{\lfloor \frac{i}{2} \rfloor,\lfloor \frac{p+1-i}{2} \rfloor \}}
\end{equation}
 if $S^3_{K}(p/2)\cong \varepsilon\mathbb T(p/q)$ and $\phi_a$ (or $\phi_{p-a}$) identifies their Spin$^c$ structures. 

\begin{proof}[Proof of Theorem~\ref{thm:TenCases}] 
We will compute the correction terms of the {\bf T}- and {\bf I}-type manifolds using (\ref{eq:TrSurgCorr}). For all $a$ satisfying (\ref{eq:aCond}), we compute the sequence $\delta^{\varepsilon}_a(i)$. Then we check if this sequence satisfies (\ref{eq:Expected}) for any $\{t_s\}$ as in (\ref{eq:tsProperties}). By Proposition~\ref{prop:pLarge}, the equality (\ref{eq:Expected}) does not hold when $p$ is sufficiently large. 
 
For small $p$, a direct computer search reveals all the possible $p/q$, which are exactly $7/2$ and the numbers given in the table in Theorem~\ref{thm:TenCases}. We also get the correction terms, from which we can recover the Alexander polynomials using (\ref{eq:Expected}) and (\ref{eq:a_i}).
By \cite{OSzLens}, we can get the knot Floer homology of the corresponding knots, which should be the knot Floer homology of either $T_{3,2}$, or $T_{5,2}$, or their cable knots as in the table in Theorem~\ref{thm:TenCases}. By Ghiggini \cite{Gh}, if the knot Floer homology is the same as that of $T_{3,2}$, then the knot must be $T_{3,2}$. So we are left with the knots corresponding to the knots in the table in Theorem~\ref{thm:TenCases}.
\end{proof}


\section{The case when $p$ is large}\label{sect:pLarge}

In this section, we will assume that
$S^3_{K}(p/2)\cong \varepsilon\mathbb T(p/q)$, and $$p=6q+\zeta r,\quad r\in\{3,5\}, \quad\varepsilon,\zeta\in\{1,-1\}.$$ 
We will prove that this does not happen when $p$ is sufficiently large. More precisely, we will show:

\begin{prop}\label{prop:pLarge}
If $p\ge192r(36r+2)^2$, then $S^3_{K}(p/2)\not\cong \varepsilon\mathbb T(p/q)$, where $p=6q+\zeta r$, $r\in\{3,5\}$.
\end{prop}

\begin{rem}
As the reader may find, the bound $192r(36r+2)^2$ can be greatly decreased by carefully improving our estimate. The first author has carried out a case-by-case analysis, which shows that $p$ cannot be greater than $6,000$. Our computer search is based on this more practical bound, rather than Proposition~\ref{prop:pLarge}.
\end{rem}

Let $s\in\{0,1,\dots,r-1\}$ be the reduction of $q$ modulo $r$. For any integer $n$, let $\theta(n)\in\{0,1\}$ be the reduction of $n$ modulo $2$, and let $\bar{\theta}(n)=1-\theta(n)$.

The equation (\ref{eq:phi}) becomes
$$
\phi_a(i)=a(i-\frac{p+1}2)+\frac{\bar{\theta}(q)p+q-1}2.
$$
Using (\ref{eq:CorrLp2}), (\ref{eq:TrSurgCorr}) and (\ref{eq:delta_a}), we get
\begin{eqnarray}
\delta^{\varepsilon}_a(i)&=&d(L(p,2),i)-\varepsilon d(\mathbb T(p/q),\phi_a(i))\nonumber\\
&=&\frac{(2i-p-1)^2}{8p}-\frac{\bar{\theta}(i)}2-\varepsilon d(L(p,q),\phi_a(i))+2\varepsilon\chi_{[0,q)}(\phi_a(i)).\label{eq:delta^eps}
\end{eqnarray}

\begin{lem}\label{lem:aEstimate}
Assume that
$S^3_{K}(p/2)\cong \varepsilon\mathbb T(p/q)$.
Let $m\in \{0,1,2,3\}$ satisfy that 
$$0\le a-mq+\frac{\bar{\theta}(q)\zeta r+q-1}2<q,$$
then
$$\left|a-\frac{mp}6\right|<\sqrt{\frac{4rp}3}.$$
\end{lem}
\begin{proof}
Since $S^3_{K}(p/2)\cong \varepsilon\mathbb T(p/q)$, there exists an integer $a$ satisfying (\ref{eq:aCond}) such that (\ref{eq:Expected}) holds.
It follows from (\ref{eq:tsProperties}) and (\ref{eq:Expected}) that 
\begin{equation}\label{eq:delta02}
\delta^{\varepsilon}_a(\frac{p+3}2)-\delta^{\varepsilon}_a(\frac{p+1}2)=0\text{ or }2.
\end{equation}

Using (\ref{eq:delta^eps}), we get
\begin{eqnarray}
&&\delta^{\varepsilon}_a(\frac{p+3}2)-\delta^{\varepsilon}_a(\frac{p+1}2)\nonumber\\
&=&\frac1{2p}-\frac{\bar{\theta}(\frac{p+3}2)}2-\varepsilon d(L(p,q),a+\frac{\bar{\theta}(q)p+q-1}2)+2\varepsilon\chi_{[0,q)}(a+\frac{\bar{\theta}(q)p+q-1}2)\nonumber\\
&&+\frac{\bar{\theta}(\frac{p+1}2)}2+\varepsilon d(L(p,q),\frac{\bar{\theta}(q)p+q-1}2)-2\varepsilon\chi_{[0,q)}(\frac{\bar{\theta}(q)p+q-1}2).\label{eq:Recurs1}
\end{eqnarray}
Let 
\begin{eqnarray*}
C_0&=&\varepsilon\zeta\left(-d(L(r,s),a-mq+\frac{\bar{\theta}(q)\zeta r+q-1}2)+d(L(r,s),\frac{\bar{\theta}(q)\zeta r+q-1}2)\right)\\
&&+\frac1{2p}+\frac12-\theta(\frac{p+1}2)+2\varepsilon\left(\chi_{[0,q)}(a+\frac{\bar{\theta}(q)p+q-1}2)-\chi_{[0,q)}(\frac{\bar{\theta}(q)p+q-1}2)\right).
\end{eqnarray*}
When $\zeta=1$, by the recursive formula (\ref{eq:CorrRecurs}), the right hand side of (\ref{eq:Recurs1}) becomes
\begin{eqnarray*}
&&\varepsilon\left(\frac{-(2a-\theta(q)p)^2+(\theta(q)p)^2}{4pq}+d(L(q,r),a-mq+\frac{\bar{\theta}(q)r+q-1}2)\right.\\
&&\quad\left.-d(L(q,r),\frac{\bar{\theta}(q)r+q-1}2)\right)\\
&&+\frac1{2p}+\frac12-\theta(\frac{p+1}2)+2\varepsilon\left(\chi_{[0,q)}(a+\frac{\bar{\theta}(q)p+q-1}2)-\chi_{[0,q)}(\frac{\bar{\theta}(q)p+q-1}2)\right)\\
&=&\varepsilon\left(-\frac{a^2}{pq}+\frac{\theta(q)a}q+\frac{(2a-2mq-\theta(q)r)^2-(\theta(q)r)^2}{4qr}\right)+C_0\\
&=&\varepsilon\left(\frac6{pr}(a-\frac{mp}6)^2-\frac{m^2-6m\theta(q)}6\right)+C_0.
\end{eqnarray*}
When $\zeta=-1$,  the right hand side of (\ref{eq:Recurs1}) becomes
\begin{eqnarray*}
&&\varepsilon\left(\frac{-(2a-\theta(q)p)^2+(\theta(q)p)^2}{4pq}+d(L(q,q-r),a-mq+\frac{-\bar{\theta}(q)r+q-1}2)\right.\\
&&\quad\left.-d(L(q,q-r),\frac{-\bar{\theta}(q)r+q-1}2)\right)\\
&&+\frac1{2p}+\frac12-\theta(\frac{p+1}2)+2\varepsilon\left(\chi_{[0,q)}(a+\frac{\bar{\theta}(q)p+q-1}2)-\chi_{[0,q)}(\frac{\bar{\theta}(q)p+q-1}2)\right)\\
&=&\varepsilon\left(-\frac{a^2}{pq}+\frac{\theta(q)a}q+\frac{(2a-2mq+\theta(q)r-q)^2-(\theta(q)r-q)^2}{4q(q-r)}\right.\\
&&\left.-d(L(q-r,r),a-mq+\frac{-\bar{\theta}(q)r+q-1}2)+d(L(q-r,r),\frac{-\bar{\theta}(q)r+q-1}2)\right)\\
&&+\frac1{2p}+\frac12-\theta(\frac{p+1}2)+2\varepsilon\left(\chi_{[0,q)}(a+\frac{\bar{\theta}(q)p+q-1}2)-\chi_{[0,q)}(\frac{\bar{\theta}(q)p+q-1}2)\right)\\
&=&\varepsilon\left(-\frac{a^2}{pq}+\frac{\theta(q)a}q+\frac{(a-mq+\theta(q)r-q)(a-mq)}{q(q-r)}-\frac{(a-mq-\bar{\theta}(q)r)(a-mq)}{(q-r)r}\right)+C_0\\
&=&\varepsilon\left(-\frac{a^2}{pq}+\frac{\theta(q)a}q-\frac{(a-mq+\theta(q)r)(a-mq)}{qr}\right)+C_0\\
&=&\varepsilon\left(-\frac6{pr}(a-\frac{mp}6)^2-\frac{m^2-6m\theta(q)}6\right)+C_0.
\end{eqnarray*}

Using (\ref{eq:CorrL3q}), (\ref{eq:CorrL5q}), we have $$|C_0|\le\frac65+\frac1{2p}+\frac12+2=\frac{37}{10}+\frac1{2p}<\frac92.$$ 
Moreover, $|\frac{m^2-6m\theta(q)}6|\le\frac32$. It follows from (\ref{eq:delta02}) that $$\left|\frac6{pr}(a-\frac{mp}6)^2\right|\le2+\frac32+\frac{9}{2}=8,$$
so our conclusion holds.
\end{proof}

\begin{lem}\label{lem:Ak+B+C}
Let $k$ be an integer satisfying 
\begin{equation}\label{eq:kRange}
0\le k<\frac1{48}\frac{p-13r+6}{\sqrt{3rp}}-\frac16.
\end{equation}
Let $$i_k=\frac{\bar{\theta}(q)p+q-1}2+a(6k)-kmp,\quad j_k=\frac{\bar{\theta}(q)\zeta r+q-1}2+a(6k)-kmp.$$
Then
$$\delta^{\varepsilon}_a(\frac{p+1}2+6k+1)-\delta^{\varepsilon}_a(\frac{p+1}2+6k)=Ak+B+C_k,$$
where 
\begin{eqnarray*}
A&=&\varepsilon\zeta\cdot\frac{2(6a-mp)^2}{pr}+\frac6p,\\
B&=&\varepsilon\left(\frac{6\zeta}{pr}(a-\frac{mp}6)^2-\frac{m^2-6m\theta(q)}6\right),\\
C_k&=&\varepsilon\zeta\big(-d(L(r,s),j_k+a-mq)+d(L(r,s),j_k)\big)\\
&&+\frac1{2p}+\frac12-\theta(\frac{p+1}2)+2\varepsilon\chi_{[0,q)}(i_k+a)-2\varepsilon\chi_{[0,q)}(i_k).
\end{eqnarray*}
\end{lem}
\begin{proof}
By  (\ref{eq:kRange}), we have
\begin{equation}\label{eq:6k+1bound}
(6k+1)\sqrt{\frac{4rp}3}<\frac{p-13r+6}{12}\le\frac{q-2r+1}2.
\end{equation}

It follows from (\ref{eq:aCond}), (\ref{eq:6k+1bound}) and Lemma~\ref{lem:aEstimate} that 
\begin{equation}\label{eq:ikjk}
0\le i_k< i_k+a<p+q,\qquad 0\le j_k,j_k+a-mq<q.
\end{equation}
For example, 
\begin{eqnarray*}
j_k+a-mq&=&j_k+a-m\frac{p-\zeta r}6\\
&=&\frac{\bar{\theta}(q)\zeta r+q-1}2+(6k+1)(a-\frac{mp}6)+\frac{m\zeta r}6\\
&<&\frac{r+q-1}2+\frac{q-2r+1}2+\frac r2\\
&=&q.
\end{eqnarray*}
Similar argument shows other inequalities.

Using (\ref{eq:delta^eps}), we can compute
\begin{eqnarray}
&&\delta^{\varepsilon}_a(\frac{p+1}2+6k+1)-\delta^{\varepsilon}_a(\frac{p+1}2+6k)\nonumber\\
&=&\frac{(6k+1)^2}{2p}-\frac{\bar{\theta}(\frac{p+1}2+6k+1)}2-\varepsilon d(L(p,q),i_k+a)+2\varepsilon\chi_{[0,q)}(i_k+a)\nonumber\\
&&-\frac{(6k)^2}{2p}+\frac{\bar{\theta}(\frac{p+1}2+6k)}2+\varepsilon d(L(p,q),i_k)-2\varepsilon\chi_{[0,q)}(i_k).\label{eq:Recurs2}
\end{eqnarray}
When $\zeta=1$,
using (\ref{eq:ikjk}) and the recursion formula  (\ref{eq:CorrRecurs}), the right hand side of (\ref{eq:Recurs2}) becomes
\begin{eqnarray*}
&&\varepsilon\left(\frac{-(2i_k+2a+1-p-q)^2+(2i_k+1-p-q)^2}{4pq}+d(L(q,r),j_k+a-mq)-d(L(q,r),j_k)\right)\\
&&+\frac{12k+1}{2p}+\frac12-\theta(\frac{p+1}2)+2\varepsilon\chi_{[0,q)}(i_k+a)-2\varepsilon\chi_{[0,q)}(i_k)\\
&=&\varepsilon\left(\frac{-a(a(12k+1)-2kmp-\theta(q)p)}{pq}+\frac{(2j_k+2a-2mq+1-q-r)^2-(2j_k+1-q-r)^2}{4qr}\right)\\
&&+\frac{6k}{p}+C_k\\
&=&\varepsilon\left(\frac{-a(a(12k+1)-2kmp-\theta(q)p)}{pq}+\frac{(a-mq)((12k+1)a-2kmp-mq-\theta(q)r)}{qr}\right)\\
&&+\frac{6k}{p}+C_k\\
&=&Ak+B+C_k.
\end{eqnarray*}
When $\zeta=-1$, the right hand side of (\ref{eq:Recurs2}) becomes
\begin{eqnarray*}
&&\varepsilon\left(\frac{-(2i_k+2a+1-p-q)^2+(2i_k+1-p-q)^2}{4pq}+d(L(q,q-r),j_k+a-mq)-d(L(q,q-r),j_k)\right)\\
&&+\frac{12k+1}{2p}+\frac12-\theta(\frac{p+1}2)+2\varepsilon\chi_{[0,q)}(i_k+a)-2\varepsilon\chi_{[0,q)}(i_k)\\
&=&\varepsilon\bigg(\frac{-a(a(12k+1)-2kmp-\theta(q)p)}{pq}+\frac{(2j_k+2a-2mq+1-2q+r)^2-(2j_k+1-2q+r)^2}{4q(q-r)}\\
&&-d(L(q-r,r),j_k+a-mq)+d(L(q-r,r),j_k)\bigg)\\
&&+\frac{12k+1}{2p}+\frac12-\theta(\frac{p+1}2)+2\varepsilon\chi_{[0,q)}(i_k+a)-2\varepsilon\chi_{[0,q)}(i_k)\\
&=&\varepsilon\bigg(\frac{-a(a(12k+1)-2kmp-\theta(q)p)}{pq}+\frac{(a-mq)(-\bar{\theta}(q)r+12ka-2kmp+a-mq-q+r)}{q(q-r)}\\
&&-\frac{(2j_k+2a-2mq+1-q)^2-(2j_k+1-q)^2}{4(q-r)r}\bigg)+\frac{6k}{p}+C_k\\
&=&\varepsilon\left(\frac{-a(a(12k+1)-2kmp-\theta(q)p)}{pq}-\frac{(a-mq)((12k+1)a-2kmp-mq+\theta(q)r)}{qr}\right)\\
&&+\frac{6k}{p}+C_k\\
&=&Ak+B+C_k.
\end{eqnarray*}
\end{proof}

\begin{proof}[Proof of Proposition~\ref{prop:pLarge}]
If $S^3_{K}(p/2)\cong \varepsilon\mathbb T(p/q)$, then (\ref{eq:Expected}) holds, so
\begin{equation}\label{eq:delta026k}
\delta^{\varepsilon}_a(\frac{p+1}2+6k+1)-\delta^{\varepsilon}_a(\frac{p+1}2+6k)=0\text{ or }2
\end{equation}
for all $k$ satisfying (\ref{eq:kRange}). If $p\ge192r(36r+2)^2$, 
then 
$$
(6\cdot 6r+1)\cdot8\sqrt{3r}\le\sqrt{p}-1<\frac{p-13r+6}{\sqrt p},
$$
hence $k=6r$ satisfies (\ref{eq:kRange}).

Let $A,B,C_k$ be as in Lemma~\ref{lem:Ak+B+C}.
If $A\ne0$, then $Ak+B+C$ is equal to 0 or $2$ for at most two values of $k$ for any given $C$.
Given $p,q,a,\varepsilon,\zeta$, as $k$ varies, $C_k$ can take at most $3r$ values.  It follows that $Ak+B+C_k$ can not be 0 or 2 for $k=0,1,\dots,6r$.
As a consequence, if $p\ge192r(36r+2)^2$, 
then (\ref{eq:delta026k}) does not hold.

The remaining case we need to consider is that
$A=0$. In this case 
\begin{equation}\label{eq:A=0}
r=3, \quad\varepsilon\zeta=-1,\quad 6a-mp=\pm3.
\end{equation}
 So 
\begin{eqnarray}
B+C_k&=&d(L(3,s),j_k+a-mq)-d(L(3,s),j_k)+\frac12-\frac{\varepsilon m^2}6\nonumber\\
&&-\theta(\frac{p+1}2)+\varepsilon \big(m\theta(q)+2\chi_{[0,q)}(i_k+a)-2\chi_{[0,q)}(i_k)\big).\label{eq:B+C}
\end{eqnarray}
Note that the second row in the above expression is always an integer.
Using (\ref{eq:CorrL3q}), the value of $d(L(3,s),j_k+a-mq)-d(L(3,s),j_k)$ is $0$ or $\pm\frac23$, so $B+C_k$ is an integer only if $m=1$ or $3$.

If $m=3$, then $a=\frac{p-1}2\equiv1\pmod3$. Since $$j_k=\frac{3\zeta\bar{\theta}(q)+q-1}2+a(6k)-kmp\equiv1-s\pmod 3,$$
it follows from (\ref{eq:CorrL3q}) that
$$d(L(3,s),j_k+a-mq)-d(L(3,s),j_k)=d(L(3,s),2-s)-d(L(3,s),1-s)=\pm\frac23,$$ so $B+C_k$ is not an integer.

If $m=1$, $a=\frac{p\pm3}6\in\{q,q+\zeta\}$. If $a=q$, $B+C_k$ is not an integer. So we must have $a=q+\zeta$. Consider the first row on the right hand side of (\ref{eq:B+C}) and let $k=0$, we get
$$d(L(3,s),1-s+\zeta)-d(L(3,s),1-s)+\frac12-\frac{\varepsilon}6.$$ 
Since this number is an integer, using (\ref{eq:CorrL3q}), we get
$$s=
\left\{
\begin{array}{ll}
2, &\text{if }\varepsilon=1,\\
1, &\text{if }\varepsilon=-1.
\end{array}
\right.$$

We consider $\delta^{\varepsilon}_a(6)-\delta^{\varepsilon}_a(7)$, which is 0 by (\ref{eq:Expected}).

If $\varepsilon=1$, it follows from (\ref{eq:A=0}) that $\zeta=-1$.
Since $a=q+\zeta=q-1$, we can compute 
$$\phi_a(6)=(q-1)(6-\frac{p+1}2)+\frac{\bar{\theta}(q)p+q-1}2\equiv6q-6\pmod p.$$
So we have
\begin{eqnarray*}
&&\delta^{+1}_a(6)-\delta^{+1}_a(7)\\
&=&\frac{(11-p)^2}{8p}-\frac{1}2-d(L(p,q),6q-6)-\frac{(13-p)^2}{8p}+d(L(p,q),q-4)-2\\
&=&\frac{p-12}{2p}-\frac12-\frac{(5q-8)^2-(-5q-4)^2}{4pq}+d(L(q,q-3),q-6)-d(L(q,q-3),q-4)-2\\
&=&-\frac{6}{p}+\frac{30q-12}{pq}+\frac{(-8)^2-(-4)^2}{4q(q-3)}-d(L(q-3,3),q-6)+d(L(q-3,3),q-4)-2\\
&=&-\frac{6}{p}+\frac{30q-12}{pq}+\frac{12}{q(q-3)}-\frac{(q-11)^2-(q-7)^2}{12(q-3)}+d(L(3,2),2)-d(L(3,2),1)-2\\
&=&-2,
\end{eqnarray*}
a contradiction.

If $\varepsilon=-1$, then $\zeta=1$.
Since $a=q+\zeta=q+1$, we get
$$\phi_a(6)=(q+1)(6-\frac{p+1}2)+\frac{\bar{\theta}(q)p+q-1}2\equiv2\pmod p.$$
So we have
\begin{eqnarray*}
&&\delta^{-1}_a(6)-\delta^{-1}_a(7)\\
&=&\frac{(11-p)^2}{8p}-\frac{1}2+d(L(p,q),2)-2-\frac{(13-p)^2}{8p}-d(L(p,q),q+3)\\
&=&\frac{p-12}{2p}-\frac12+\frac{(5-p-q)^2-(2q+7-p-q)^2}{4pq}-d(L(q,3),2)+d(L(q,3),3)-2\\
&=&-\frac{6}{p}+\frac{(q+1)(p-6)}{pq}-\frac{(5-q-3)^2-(7-q-3)^2}{12q}+d(L(3,1),2)-d(L(3,1),0)-2\\
&=&-2,
\end{eqnarray*}
a contradiction.
\end{proof}

\end{document}